\documentclass[11pt,letterpaper]{article} 

\usepackage{amsmath, amssymb, amsthm, graphicx} 
\usepackage[margin=1.0in]{geometry} 
\usepackage{mathrsfs} 
\usepackage{hyperref} 
\usepackage{enumitem} 

\newtheorem{theorem}{Theorem}
\newtheorem{lemma}[theorem]{Lemma}
\newtheorem{proposition}[theorem]{Proposition}
\newtheorem{corollary}[theorem]{Corollary}

\theoremstyle{definition}
\newtheorem{definition}[theorem]{Definition}
\theoremstyle{remark}

\newcommand{\R}{\mathbb{R}}
\newcommand\per{\textup{per}}            
\newcommand\bdd{\textup{b}}              
\newcommand{\patm}{P_{\textup{atm} }}    
\newcommand{\cm}{\mathscr C}             
\newcommand{\cms}{\mathscr C^*}          
\newcommand{\w}{\mathscr W}              
\newcommand{\qt}{\mathscr Q}             
\newcommand{\K}{\mathscr K}
\newcommand{\trough}{(L,\eta(L))}        
\newcommand{\crest}{(0,\eta(0))}         
\newcommand{\lc}{\lambda_{\textup{cr} }} 
\newcommand\pa{\partial}
\newcommand\loc{{\textup{loc}}}
\DeclareMathOperator{\re}{\mathrm{Re}}   
\renewcommand{\Re}{\re}

\newcommand\mx[1]{#1_{\textup{max} }} 
\newcommand\mn[1]{#1_{\textup{min} }} 

\newcommand{\abs}[2][]{#1\lvert #2 #1\rvert}

\title{Bound on the slope of steady water waves with favorable vorticity}
\author{
  Walter A.~Strauss\thanks{Research supported in part by NSF grant DMS-1007960.}
  \and
  Miles H.~Wheeler\thanks{Research supported in part by NSF grant DMS-1400926.}
}
\date{February 20, 2015}
\begin{document}
\maketitle

\begin{abstract}
  We consider the angle $\theta$ of inclination (with respect to the
  horizontal) of the profile of a steady 2D inviscid symmetric
  periodic or solitary water wave subject to gravity.  Although
  $\theta$ may surpass $30^\circ$ for some irrotational waves close to
  the extreme wave, Amick \cite{amick:bounds} proved that for any
  irrotational wave the angle must be less than 31.15$^\circ$. Is the
  situation similar for periodic or solitary waves that are not
  irrotational? The extreme Gerstner wave has infinite depth, adverse
  vorticity and vertical cusps ($\theta = 90^\circ$).  Moreover,
  numerical calculations show that even waves of finite depth can
  overturn if the vorticity is adverse.   In this paper, on the other
  hand,  we prove an upper bound of 45$^\circ$ on $\theta$ for a large
  class of waves with favorable vorticity and finite depth.  In
  particular, the vorticity can be any constant with the favorable
  sign.  We also prove a series of general inequalities on the
  pressure within the fluid, including the fact that any overturning
  wave must have a pressure sink.  
\end{abstract}

\section{Introduction} \label{sec:intro}

The celebrated ``extreme Stokes wave" has angle (with respect to the
horizontal) exactly 30$^\circ$ at its crest, as originally conjectured
by Stokes himself \cite{stokes}. This is the limiting wave, singular
at its crest, of the curve $\K$ of irrotational waves that bifurcates
from a trivial (flat laminar) wave.  However, it was a surprise when
numerical calculations \cite{cs:numericalnew,sm:limiting} indicated
that the angle $\theta$ may surpass $30^\circ$ for some waves on $\K$
that are very close to the extreme wave.  This fact was subsequently
proven by McLeod in \cite{mcleod:stokes}.  The maximum angle does not
occur at the crest but very close to it.  In a remarkable paper
\cite{amick:bounds} Amick proved that for any irrotational wave the
angle must be less than 31.15$^\circ$.

For waves that are not irrotational, there are no known analytical
bounds on $\theta$, even 90$^\circ$. Indeed, with adverse vorticity,
the crests of Gerstner's explicit waves in deep water can have cusps
with a 90$^\circ$ angle, the extreme case being a cycloid, and numerical
calculations in finite depth show that waves can be quite steep and
even overturn \cite{sp:steep, KS, VO}. In all of these cases, the
vorticity is \emph{adverse} (positive in our formulation). On the
other hand, the main purpose of the present paper is to prove an upper
bound of 45$^\circ$ for a large class of waves with {\it favorable}
vorticity (negative  in our formulation) in finite depth. The
waves to which our bound applies include at least a large part of the
well-known bifurcating curve.  

We denote the velocity by $(U,V)$, the vorticity by $\omega=V_x-U_y$,
and the (relative) stream function by $\psi$, where $\psi_y=U-c,\
\psi_x=V$. The vorticity $\omega$ depends only on the stream function,
$\omega(x-ct,y)=\gamma(\psi(x-ct,y))$. By a streamline we mean a level
curve of $\psi$, i.e.~a particle path {\it in a frame moving with the wave}.
Our main result, somewhat informally stated, is as follows.

\begin{theorem} \label{Theorem1} 
  Let $\cm$ be a connected set (containing a trivial wave) of
  symmetric periodic finite-depth water waves traveling at speed $c$
  with a single crest and trough per period for which $U<c$
  (non-stagnation) and for which the streamlines (in the moving frame) are strictly
  decreasing from crest to trough. We assume that $\gamma\le0,\
  \gamma'\le0,\ \gamma''\le0$.   At least until $(U-c)\gamma=g$ at the
  troughs, the waves in $\cm$ that bifurcate from a trivial wave have
  angle strictly less than $45^\circ$. (The trivial wave satisfies
  $(U-c)\gamma<g$ everywhere.) 
 
  The same statement is true for symmetric solitary waves (instead of
  periodic waves) at least until $(U_\infty-c)\gamma=g$ on the
  surface, where $U_\infty(y)=\lim_{x\to\pm\infty} U(x,y)$.  
\end{theorem}

In fact, we will prove a bound that is strictly less than $45^\circ$,
although the bound depends on the wave; see Corollaries~\ref{cor:sigma} and
\ref{cor:sigma:sol}.

In \cite{amick:bounds}, Amick first shows that $\theta < \pi/2\beta
\approx 38.2^\circ$, where $\beta = (9+\sqrt{97})/8$, and then uses a
different argument to improve this bound to $31.15^\circ$. His first
step is somewhat related to our arguments for waves with vorticity,
and we outline it now. Working with the Nekrasov formulation, which is
valid only in the irrotational case, Amick essentially considers 
\begin{align*}
  f_\alpha = \Re[((c-U)+iV)^\alpha] 
  = [(U-c)^2+V^2]^{\alpha/2} \cos(\alpha\theta),
\end{align*}
which is the real part of an analytic function of the complex variable
$z=x+iy$, where $\alpha \ge 1$ is a parameter. The bound $\abs \theta
< \pi/2\alpha$ will follow if we can show that $f_\alpha > 0$. Since
$f_\alpha = (c-U)^\alpha > 0$ at crests and troughs, it is enough for
$f_\alpha$ to be decreasing from crest to trough along the free
surface, i.e.~for its tangential derivative
\begin{align*}
  W_\alpha = \frac{(U-c)\pa_x + V\pa_y}{(U-c)^2+V^2} f_\alpha,
\end{align*}
which is again the real part of an analytic function of $z$, to have
an appropriate sign on each half period. By combining maximum
principle arguments with a continuation in the parameter $\alpha$,
Amick deduces that $W_\alpha$ has the appropriate sign (and hence
$\abs\theta < \pi/2\alpha$) for $1 \le \alpha \le \beta$. Toland and
Plotnikov \cite{pt:fourier} subsequently gave a different proof, also
depending on the Nekrasov formulation and complex variable techniques,
that the angle is less than 45$^\circ$. 

There are considerable difficulties in extending Amick's arguments to
waves with vorticity. First, with a general vorticity, the water wave
problem cannot be reformulated as a Nekrasov-type integral equation on
the boundary. More significantly, complex function theory no longer
guarantees that the functions $f_\alpha$ and $W_\alpha$ are harmonic.
Indeed, for a general $\alpha$, formulas for their Laplacians contain
many terms of seemingly indeterminate sign. Thus Amick's application
of the maximum principle to $W_\alpha$ seems out of reach.
Nevertheless, we are able to distill some of Amick's ideas in the
special case $\alpha = 2$, where the formulas are simpler, and under
the additional assumption $\gamma,\gamma',\gamma'' \le 0$ on the
vorticity function. Instead of varying $\alpha$, we use a continuation
argument along a connected set $\cm$ of solutions, an option which
Amick also explores in \cite{amick:bounds}. 

Section \ref{sec:periodic} is devoted to the main theorem (Theorem
\ref{thm:main}), which is a more precise version of Theorem \ref{Theorem1}  
that specifies the notion of connectedness.  The key to the proof is
that $U_x$ has a strict sign between each crest and trough.  In order
to prove this fact, we look at the possibility that, for some wave
along $\cm$, $U_x$ might vanish somewhere away from the crest and
trough.  Then we make use of the Hopf maximum principle applied to
$U_x/(U-c)$ and the slope $V/(U-c)$ together with the boundary
conditions to prove certain inequalities.  The quantity
$g-(U-c)\gamma$ plays an important role. We note that the product
$(U-c)\gamma$ is the vertical component of what is sometimes called
the ``vortex force'' \cite{saffman:dynamics}. 

Construction of the connected bifurcation set goes back to
\cite{kn:existence} in the irrotational case and \cite{cs:exact} in
the case of general vorticity. In Section \ref{sec-existence} we
discuss the relationship between our basic result and the bifurcation
set of solutions that we know exist.  

In Section \ref{sec-solitary}, Theorem \ref{Theorem1} is proven for 
the case of solitary waves.  

Section \ref{sec-pressure} is devoted to a series of inequalities on
the pressure inside the fluid in the absence of stagnation points.
Some of them have been previously known and others are new facts but
in any case we prove them in enough generality to permit overturning
waves. All of them are based on the maximum principle applied to
various quantities.  One of them is used in the proof of the main
theorem.  Allowing for stagnation points inside the fluid but not on
the free surface, we also prove that every overturning wave must have
a pressure sink. 

We thank John Toland for informing us of the reference \cite{pt:fourier}.

\section{Bound on the slope} \label{sec:periodic}
{\it For the rest of this paper it is convenient to denote the
relative velocity by $u=U-c$ and $v=V$.} Except in
Sections~\ref{sec-solitary} and \ref{sec-pressure:solitary}, we
consider symmetric $2L$-periodic water waves with vorticity. In a
frame moving with the wave they are described as solutions of 
\begin{subequations} \label{eqn:ww}
  \begin{alignat}{2}
    \label{eqn:ww:u}
    uu_x + vu_y &= -P_x &\quad& \textup{ in } {-d < y < \eta(x)},\\
    \label{eqn:ww:v}
    uv_x + vv_y &= -P_y - g &\quad& \textup{ in } {-d < y < \eta(x)},\\
    \label{eqn:ww:div}
    u_x + v_y &= 0 &\quad& \textup{ in } {-d < y < \eta(x)},\\
    \label{eqn:ww:dynamic}
    P &= \patm  &\quad& \textup{ on } y = \eta(x),\\
    \label{eqn:ww:top}
    v &= \eta_x u  &\quad& \textup{ on } y = \eta(x),\\
    \label{eqn:ww:bot}
    v &= 0  &\quad& \textup{ on } y = -d,
  \end{alignat}
\end{subequations}
with $u,\eta$ even in $x$, $v$ odd in $x$, and where $u,v,\eta,P$ are
$2L$-periodic in the $x$ variable. We take $\eta(x)$ to have mean
value zero, so that $d$ is the average depth.  The ``atmospheric''
pressure $\patm$ and depth $d > 0$ are constants. As for regularity,
we will assume $\eta \in C^4_\per[-L,L]$ and $u,v,P \in
C^3_\per(\overline{\Omega})$, where
\begin{align*}
  \Omega = \{(x,y) \in \R^2 : -L < x < L,\ -d < y < \eta(x)\} 
\end{align*}
denotes a period the fluid domain and ``$\per$'' denotes
$2L$-periodicity in $x$. In addition, we will always assume that
\begin{align}
  \label{eqn:nostag}
  \sup_\Omega u < 0,
\end{align}
which in particular rules out stagnation points in the fluid where
$u=v=0$.

We say a wave is \emph{trivial} if $u,v,\eta,P$ depend only on the
vertical variable $y$. This forces $v \equiv 0$, $\eta \equiv 0$, and
$P = -gy$, but does not place any new restrictions on the horizontal
velocity $u =u_0(y)$, which at this point can be an arbitrary
(negative) function of $-d \le y \le 0$.

Let $\Omega^+$, $S^+$, and $B^+$ denote half-periods of
the fluid domain, free surface, and bed,
\begin{align*}
  \Omega^+ &= \{(x,y) \in \R^2 :  0 < x < L,\ -d < y < \eta(x)\},\\
  S^+ &= \{(x,\eta(x)) : 0 < x < L \},\\
  B^+ &= \{(x,-d) : 0 < x < L \}.
\end{align*}
See Figure~\ref{fig:omegaplus}a.  We assume that all nontrivial waves
are strictly monotone in that
\begin{align}
  \label{eqn:mono}
  v > 0 \textup{ in } \Omega^+ \cup S^+.
\end{align}
Using \eqref{eqn:nostag} and \eqref{eqn:mono} in \eqref{eqn:ww:top},
we see that the slope 
\begin{align*}
  \eta_x = \frac vu < 0 \quad \text{ on } S^+. 
\end{align*}
That is, the free surface $\eta$ is strictly decreasing between the
crest at $\crest$ and trough at $\trough$. All of the above
assumptions are satisfied by the waves constructed in \cite{cs:exact}.

\begin{figure}[t]    
  \centering   
  \includegraphics[scale=1.1]{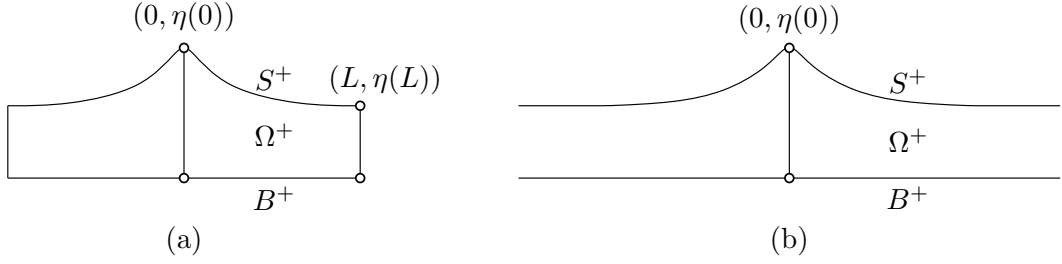}
  \caption{(a) The half-period $\Omega^+$ of the fluid domain, together
  with the portions $S^+$ and $B^+$ of its boundary on the free
  surface and bed. (b) Analogous definitions for a solitary wave.}  
  \label{fig:omegaplus}  
\end{figure}

Thanks to incompressibility \eqref{eqn:ww:div}, there exists a stream
function $\psi \in C^4_\per(\overline\Omega)$, unique up to an
additive constant, satisfying
\begin{align*}
  \psi_y = u,
  \qquad \psi_x = -v
\end{align*}
in the fluid domain $\Omega$. By the kinematic boundary conditions
\eqref{eqn:ww:top} and \eqref{eqn:ww:bot}, $\psi$ is constant on both
the free surface $y=\eta(x)$ and the bed $y=-d$. We normalize $\psi$
so that it vanishes on the free surface, and let $m$ denote its value
on $y=-d$, which is positive since by \eqref{eqn:nostag} we have
$\psi_y = u < 0$. The assumption \eqref{eqn:nostag} also guarantees
that the vorticity $\omega = v_x - u_y$ can be expressed as 
\begin{align*}
  \omega = -\Delta \psi = \gamma(\psi) 
\end{align*}
for some function $\gamma \in C^2[0,m]$ called the \emph{vorticity
function}; see for instance \cite{cs:exact}. 

The main additional assumption that we will make throughout
Section~\ref{sec:periodic} is that the vorticity function $\gamma$
satisfies the sign conditions
\begin{align}
  \label{eqn:mastergam}
  \gamma(\psi) \le 0
  ,\ 
  \gamma'(\psi) \le 0
  ,\ 
  \gamma''(\psi) \le 0
  \quad \textup{ for } 0 < \psi < m.
\end{align}
Note that \eqref{eqn:mastergam} is satisfied whenever $\gamma \le 0$
is constant. For the case of a trivial wave with $u=u_0(y) < 0$, 
a short calculation shows that 
\eqref{eqn:mastergam} is equivalent to 
\begin{align}
  \label{eqn:masteru0}
  u_{0y} \ge 0,
  \qquad 
  u_{0yy} \le 0,
  \qquad 
  u_0 u_{0yyy} - u_{0y} u_{0yy} \le 0.
\end{align}

\begin{definition}[The water wave set $\w$] 
  Fixing the half-period $L$, we denote by $\w$ the set of waves
  satisfying the above assumptions, or more precisely the set of
  5-tuples $(u,v,P,\eta,d)$ with $d > 0$, $\eta \in C^4_\per[-L,L]$
  and $u,v,P \in C^3_\per(\overline\Omega)$ for which $u,P,\eta$ are
  even in $x$ and $v$ is odd in $x$, which satisfy \eqref{eqn:ww},
  \eqref{eqn:nostag}, and \eqref{eqn:mastergam}, and which either
  satisfy the monotonicity condition \eqref{eqn:mono} or are trivial.
\end{definition}

We are interested in subsets $\cm \subset \w$ which are {\it
connected} in some sense. In \cite{cs:exact}, this connectedness is
expressed by first making a change of variables which transforms the
problem into a periodic strip. For our purposes, a weaker notion of
connectedness merely involving traces on the free surface will
suffice. Consider the mapping
\begin{align}
  \label{eqn:tau}
  \tau \colon \w \to (C^0[0,L])^7,
  \qquad 
  \tau(u,v,P,\eta,d)(x) = (u,v,u_x,v_x,u_y,u_{xx},u_{xy})(x,\eta(x)).
\end{align}
Our main result below is a more precise version of Theorem 1. It
concerns subsets $\cm$ of $\w$ which contain a trivial wave and for
which  $\tau(\cm)$ a connected subset of $(C^0[0,L])^7$. One can check
that the sets $\cm$ constructed in \cite{cs:exact} satisfy this
connectedness property. 

\begin{theorem}\label{thm:main} 
  Let $\cm \subset \w$ be a set of periodic water waves such that
  $\tau(\cm)$ is a connected subset of $(C^0[0,L])^7$. If $\cm$
  contains at least one trivial wave, and if all the waves in $\cm$
  satisfy
  \begin{align}
    \label{eqn:trough}
    g-\gamma u > 0 \quad \textup{at the trough $\trough$},
  \end{align}
  then the slope $\abs{v/u}$ of any streamline of any wave in $\cm$ is
  $<1$.
\end{theorem}

The proof of Theorem \ref{thm:main} depends on a couple of lemmas.
The first one provides a sufficient condition for the upper bound on
the slope.  

\begin{lemma}\label{lem:easy}
  If a wave in $\w$  satisfies $u_x < 0$ in $\Omega^+ \cup S^+$, 
  then $\max_{\overline\Omega} \abs{v/u} < 1$.
  \begin{proof}
    Letting $f = \frac 12(u^2 - v^2)$, we first note that the desired
    inequality is equivalent to $f > 0$ on $\overline{\Omega^+}$. By
    symmetry and periodicity, $v=0$ on the vertical lines $x=0$ and
    $x=L$ as well as the bed $B^+$. Since $\max_{\overline\Omega} u <
    0$ by \eqref{eqn:nostag}, we therefore have $f=u^2/2 > 0$ on
    $\pa\Omega^+\setminus S^+$.   
    Differentiating, we find 
    \begin{align}
      \label{eqn:hasasign}
      uf_x + vf_y = (u^2+v^2)u_x - \gamma uv < 0
      \quad \text{ on } \Omega^+ \cup S^+,
    \end{align}
    where we have used the identities $u_x+v_y=0$ and $v_x-u_y=\gamma$
    to eliminate derivatives of $v$ in favor of derivatives of $u$.
    The expression in \eqref{eqn:hasasign} is strictly negative in
    $\Omega^+ \cup S^+$, since $u_x < 0$ by assumption, and $u < 0$,
    $v\ge 0$, and $\gamma \le 0$ by \eqref{eqn:nostag},
    \eqref{eqn:mono}, and \eqref{eqn:mastergam}.
    
    So suppose for the sake of contradiction that
    $\min_{\overline{\Omega^+}} f \le 0$.  Then by the previous
    comments the minimum could not be achieved on
    $\pa\Omega^+\setminus S^+$ but only at some point $(x^*,y^*)$ in
    $\Omega^+ \cup S^+$.   Now if $(x^*,y^*)$ lies in $\Omega^+$, then
    it would be a critical point of $f$; that is, $f_x = f_y = 0$ at
    $(x^*,y^*)$. But then the left hand side of \eqref{eqn:hasasign}
    would vanish, which is a contradiction.   Finally if $(x^*,y^*)$
    lies on the surface $S^+$, then we would have 
    \begin{align*}
      0 > (u^2+v^2)u_x - \gamma uv
      = uf_x + vf_y = u \frac d{dx} f(x,\eta(x)) = 0
      \qquad \textup{ at }(x^*,y^*),
    \end{align*}
    which is again a contradiction.  Thus the minimum of $f$ is positive.  
  \end{proof}
\end{lemma}

Now let $\w_0$ denote the set of waves in $\w$ for which $u_x \le 0$
on $S^+$. The set $\w_0$ certainly contains all the trivial waves in
$\w$ because they satisfy $u_x \equiv 0$.   Looking back at the
definition \eqref{eqn:tau} of the mapping $\tau$, we also see that a
wave $(u,v,P,\eta,d)$ in $\w$ lies in $\w_0$ if and only if
$\tau(u,v,P,\eta,d)$ lies in $\tau(\w_0)$. Our main tool to prove
Theorem~\ref{thm:main} is the following lemma. 

\begin{lemma}[Bounds along $S^+$]\label{lem:bound}
  \hfill
  \begin{enumerate}[label=\textup{(\alph*)}]
  \item\label{lem:bound:le} For every nontrivial wave in $\w_0$, $u_x
    < 0$ on $\Omega^+ \cup S^+$.
  \item\label{lem:bound:gtr} For every nontrivial wave in $\w
    \setminus \w_0$, $\inf_{\Omega^+} u_x/u < 0$ is achieved at a
    point $z_0 \in S^+$ where
    \begin{align}
      \label{lem:bound:bound}
      g(u^4-4u^2v^2 -v^4) - \gamma u^3(u^2 + 5v^2) < 0.
    \end{align}
  \end{enumerate}
\end{lemma}

\begin{proof}  
  Note that there are no derivatives in the expression
  \eqref{lem:bound:bound}. We begin by writing the identities that
  are obtained by differentiating the dynamic boundary condition
  \eqref{eqn:ww:dynamic} once and then twice along the free surface
  $S$. Since the pressure $P$ is constant along the free surface $S$,
  we have 
  \begin{align*}
    0 &= u \frac d{dx} P(x,\eta(x))
    = uP_x + vP_y
    = -u(uu_x+vu_y)+v(uv_x+vv_y +g) 
    \qquad \textup{on $S$},
  \end{align*}
  where have used the kinematic boundary condition \eqref{eqn:ww:top}
  to eliminate $\eta_x$ in favor of $v$ and then the Euler equations
  \eqref{eqn:ww:u}--\eqref{eqn:ww:v} to eliminate $P_x$ and $P_y$.
  Using incompressibility $u_x+v_y=0$ to eliminate $v_y$ and the
  definition $v_x-u_y=\gamma$ of the vorticity function $\gamma$ to
  eliminate $v_x$, we are left with
  \begin{align}
    \label{eqn:pone}
    (v^2-u^2)u_x - 2uv u_y - gv - \gamma uv = 0
    \qquad \textup{on $S$}.
  \end{align}
  Taking another derivative along the free surface, we find
  \begin{align}
    \label{eqn:ptwo}
    \left.
    \begin{aligned}
      0 &= (u\pa_x + v\pa_y)\left[(v^2-u^2)u_x - 2uv u_y - gv - \gamma uv
      \right]\\
      &= -2(u^2+v^2)(u_x^2+u_y^2) + gvu_x - guu_y + (3uv^2 - u^3) u_{xx} +
      (v^3-3u^2 v)u_{xy}
      \\
      &\qquad -\gamma((3u^2+v^2) u_y - 2uv u_x + gu)
      +2 \gamma' u^2 v^2 - \gamma^2 u^2  
    \end{aligned}
    \right\}
    \textup{ on $S$}, 
  \end{align}
  where now we have also differentiated the identities $u_x+v_y=0$ and
  $v_x - u_y = \gamma$ to eliminate the second partials of $v$.
  
  Consider any nontrivial wave in $\w$.  We will apply maximum principle
  arguments to the two functions
  \begin{align*}
    w:= \frac{u_x}u,
    \qquad s := \frac vu.
  \end{align*}
  The second one is the slope of the streamlines.  Now two tedious but
  elementary computations show that both of these functions satisfy
  elliptic equations, namely, 
  \begin{align}
    \label{eqn:w}
    &\Delta w
    + 2\frac{u_x}u w_x
    + 2\frac{u_y}u w_y
    = \gamma'' v \le 0,\\
    \label{eqn:s}
    &\Delta s 
    + 2v\frac{\gamma -u s_x}{u^2+v^2} s_x
    + 2u\frac{\gamma -v s_y}{u^2+v^2} s_y
    =0.
  \end{align}
  Therefore the maximum principle implies that 
  \begin{align*}
    I := \inf_{\Omega^+} w = \inf_{\pa\Omega^+} w,
    \qquad 
    \sup_{\Omega^+} s = \sup_{\pa\Omega^+} s  
  \end{align*}
  are not attained in $\Omega^+$.   
  Now $v > 0$ and hence $s < 0$ on $\Omega^+$ by \eqref{eqn:mono}. 
  By symmetry, $s=v=0$ on $B^+$.   
  Thus the Hopf lemma (strong maximum principle) yields the inequality 
  \begin{align*}
    0 > s_y = - \frac{u_x}u = -w
    \qquad \text{ on } B^+,
  \end{align*}
  so that  $w > 0$ on the bottom $B^+$.   
  On the lateral boundaries $x=0$ and $x=L$, we have $w=0$ by symmetry.  
  Thus $I\le 0$ and  $w \ge 0$ on all of $\pa \Omega^+ \setminus S^+$.  

  Suppose now that the minimum $I$ is achieved on the surface $S^+$,
  say at some point $z_0 := (x_0,\eta(x_0)) \in S^+$. Then the
  tangential derivative $\pa_x (w(x,\eta(x)))$  must vanish at
  $x=x_0$.  Thus we have 
  \begin{align}
    \label{eqn:tang}
    0 =  uw_x + vw_y = 
    - uw^2 - \frac vu u_yw + \frac vu u_{xy} + u_{xx}
    \qquad \textup{ at } z_0.
  \end{align}
  By the Hopf lemma we also know that 
  \begin{align}
    \label{eqn:hopf}
    w_y(z_0) < 0
    \qquad \textup{ at } z_0.
  \end{align}
  Solving \eqref{eqn:tang} for $u_{xx}$ and \eqref{eqn:pone} for $u_y$
  and plugging these values into \eqref{eqn:ptwo} at the point $z_0$,
  we obtain the equation 
  \begin{align}
    \label{eqn:conc}
    \begin{aligned}
      w_y &= \frac{u_{xy}}u - \frac{u_y}u w\\
      &= -\frac{u(u^2+v^2)^2}{4v^3} w^2
      + \frac{g(u^4-4u^2v^2 -v^4) - \gamma u^3(u^2 + 5v^2)}
      {4u^2 v^2(u^2+v^2)} w
      - v \frac{g^2+ g\gamma u - 4\gamma' u^4}{4u^3(u^2+v^2)}
      \\
      &:= Aw^2 + Bw + C 
    \end{aligned}
  \end{align}
  at $z_0$, where the coefficients $A,B,C$ are functions of $u,v$ and
  $\gamma$ evaluated at $z_0$. Note that  $v > 0$ at $z_0 \in S^+$ thanks
  to our monotonicity assumption \eqref{eqn:mono}. Combining
  \eqref{eqn:conc} and \eqref{eqn:hopf}, we have
  \begin{align}
    \label{eqn:useme}
    Aw(z_0)^2 + Bw(z_0) + C < 0.
  \end{align}
  Since $\gamma(0), \gamma'(0) \le 0, u < 0$ and $v>0$, 
  both of the coefficients $A$ and $C$ are strictly positive. 
  In particular, $w(z_0)\ne0$ and hence $w(z_0) < 0$.

  Consider a wave in $\w_0$.   Then $u_x\le0$ and hence $w\ge0$ and
  $I=0$. However, we have just shown that $I=0$ cannot be attained on $S^+$,
  since at such a point $z_0 \in S^+$ we would have to have $w(z_0) <
  0$. This completes the proof of \ref{lem:bound:le}. 

  On the other hand for a wave in $\w\setminus \w_0$, we have shown
  that $I<0$ is attained at some point $z_0\in S^+$ where $w(z_0) < 0$
  and where \eqref{eqn:useme} holds. The first and last terms in
  \eqref{eqn:useme} being strictly positive, the middle term $Bw(z_0)$
  must be strictly negative.  Therefore $B > 0$.  This is exactly the
  same as the inequality \eqref{lem:bound:bound} in
  \ref{lem:bound:gtr}.
\end{proof}

\begin{proof}[Proof of Theorem~\ref{thm:main}]
  Thanks to Lemma~\ref{lem:easy} and
  Lemma~\ref{lem:bound}\ref{lem:bound:le}, the desired bound
  $\abs{v/u} < 1$ holds for any wave in $\w_0$. Thus to prove the
  theorem it suffices to show that $\cm \subset \w_0$, or equivalently
  $\tau(\cm) \subset \tau(\w_0)$. Now $\tau(\cm)$ is connected, so to
  prove $\tau(\cm) \subset \tau(\w_0)$ it is enough to show that
  $\tau(\cm) \cap \tau(\w_0)$ is nonempty, relatively open, and
  relatively closed in $\tau(\cm)$. It is easy to see from the
  nonstrict inequality in the definition of $\w_0$ that $\tau(\cm)
  \cap \tau(\w_0)$ is relatively closed, and it is nonempty since the
  trivial wave in $\cm$ lies in $\w_0$. Thus it remains to show that
  $\tau(\cm) \cap \tau(\w_0)$ is relatively open in $(C^0[0,L])^7$.

  Assume the contrary.  That is, there exists a sequence
  $\tau(u_n,v_n,P_n,\eta_n,d_n)$ in $\tau(\cm) \setminus \tau(\w_0)$
  for which 
  \begin{align}
    \label{eqn:conv}
    \tau(u_n,v_n,P_n,\eta_n,d_n) \longrightarrow \tau(u,v,P,\eta,d) 
    \quad 
    \text{ in $(C^0[0,L])^7$ }
  \end{align}
  for some $\tau(u,v,P,\eta,d)$ in $\tau(\cm) \cap \tau(\w_0)$. This
  means that $u_n$, $v_n$, and their first and second partials all
  converge uniformly as functions of $x$ along the free surface. Since
  the definitions of $\tau$ and $\w_0$ imply $\tau(\cm) \setminus
  \tau(\w_0) = \tau(\cm\setminus \w_0)$ and $\tau(\cm) \cap \tau(\w_0)
  = \tau(\cm \cap \w_0)$, we have $(u_n,v_n,P_n,\eta_n,d_n) \in \cm
  \setminus \w_0$ and $(u,v,P,\eta,d) \in \cm \cap \w_0$. Thus
  $(u_n)_x \not\le 0$ on $S^+_n$ but $u_x\le0$ on $S^+$.  
  
  Applying Lemma~\ref{lem:bound}\ref{lem:bound:gtr} to
  $(u_n,v_n,P_n,\eta_n,d_n)$ for any fixed $n$, we know that the function
  $u_{nx}/u_n$ achieves a negative minimum over
  $\overline{\Omega^+_n}$ at some point $z_n = (x_n,\eta_n(x_n))\in
  S^+_n$, with $0 < x_n < L$, at which  we have 
  \begin{align}
    \label{eqn:takeme}
    \Big[g(u^4_n-4u^2_nv^2_n -v^4_n) 
    - \gamma_n u^3_n(u^2_n + 5v^2_n)\Big](z_n) < 0.
  \end{align}
  By compactness we can assume that $z_x$ converges to some $x_0$ with
  $0 \le x_0 \le L$. Set $z_0 = (x_0,\eta(x_0))$. The uniform
  convergence \eqref{eqn:conv} implies that 
  \begin{align}
    \label{eqn:cleanconv}
    (u_n,v_n,u_{nx},v_{nx},u_{ny},u_{nxx},u_{nxy})(z_n)
    \longrightarrow (u,v,u_x,v_x,u_y,u_{xx},u_{xy})(z_0)
  \end{align}
  as $n\to \infty$.   
  Letting $\gamma_n$ be the vorticity function of
  $(u_n,v_n,P_n,\eta_n,d_n)$, we also have 
  \begin{align*}
    \gamma_n(0)
    = (v_{nx}-u_{ny})(z_n) 
    \to (v_x-u_y)(z_0) 
    = \gamma(0).
  \end{align*}

  In case $(u,v,P,\eta,d)$ is a trivial wave, with $\eta \equiv 0$, $v
  \equiv 0$, $u(x,y) = u_0(y)$, then \eqref{eqn:cleanconv} would imply
  $u_n(z_n) \to u_0(0) < 0$ and $v_n(z_n) \to 0$ as $n \to \infty$.
  Taking $n \to \infty$ in \eqref{eqn:takeme} therefore yields
  $gu_0^4(0)-\gamma u^5_0(0) \le 0$. Factoring and recalling that
  $u_0(0) < 0$, we deduce that $g-\gamma u_0(0) \le 0$. Since $\gamma$
  and $u=u_0$ are independent of $x$, this in particular implies
  $g-\gamma u \le 0$ at the trough $\trough$, which contradicts
  \eqref{eqn:trough}.
 
  So $(u,v,P,\eta,d)$ must be nontrivial. Because 
  \begin{align*}
    0 < u_{nx}(z_n) \to u_x(z_0) \le 0,
  \end{align*}
  we must have $u_x(z_0) = 0$. But
  Lemma~\ref{lem:bound}\ref{lem:bound:le} asserts that $u_x(x,\eta(x))
  < 0$ for  all $0 < x < L$. Thus the only remaining possibilities are
  that $z_0$ is the crest $\crest$ or the trough $\trough$. In either
  case, $v_n(z_n) \to v(z_0) = 0$ and $u_n(z_n) \to u(z_0) < 0$ as $n
  \to \infty$.  Sending $n\to\infty$ in  \eqref{eqn:takeme}, we get
  $u^4(g-\gamma u) \le 0$ at $z_0$, which  yields $g-\gamma u(z_0) \le
  0$. Because of the assumption  \eqref{eqn:trough}, $z_0$ cannot be
  the trough.  So $z_0$ must be the crest.   
  
  By Theorem~\ref{thm:pressure}\ref{thm:pressure:top} below, we know
  that $\pa P/\pa n <0$ along the free surface. Since $\eta_x < 0$ for
  $0 < x < L$, this  in particular implies that $P_x < 0$ on $S^+$.
  But then by the basic equations \eqref{eqn:ww:u} and \eqref{eqn:ww:top},
  we have 
  \begin{align*}
    \frac d{dx} u^2(x,\eta(x)) = 2uu_x + 2vu_y =  -2{P_x} > 0 
    \quad \text{for $0 < x < L$}. 
  \end{align*}
  This means that $u$ is strictly monotone along $S^+$. Thus
  $0> u\crest > u\trough$, so that 
  \begin{align*}
    g-\gamma u\trough 
    \le g - \gamma u \crest 
    = g-\gamma u(z_0)
    \le 0,
  \end{align*}
  contradicting \eqref{eqn:trough}. This completes the proof.  
 \end{proof}

\begin{corollary}
  Theorem~\ref{thm:main} remains true if \eqref{eqn:trough} is
  replaced by $u_{xx} \ne 0$ at the trough for nontrivial waves.
\end{corollary}
\begin{proof}
  Consider $(u_n,v_n,P_n,\eta_n,d_n)$, $(u,v,P,\eta,d)$, and $z_n =
  (x_n,\eta_n(x_n))$ as in the proof of Theorem~\ref{thm:main}, and
  assume that $(u,v,P,\eta,d)$ is nontrivial and satisfies $u_{xx} \ne
  0$ at the trough. 
  Following the preceding argument we deduce that $x_n \to x_0$, 
  $z_0$ is either the crest or the trough, $u_x(z_0)=v(z_0)=0$, and
  $g-\gamma u(z_0) \le 0$. 
  Since the functions $u_{nx}/u_n$ are minimized at $z_n$, we also
  have
  \begin{align*}
    0 = (u_n\pa_x + v_n\pa_y) \frac{u_{nx}}{u_n}
    = u_n^{-2} (u_nv_nu_{nxy} - v_nu_{nx} u_{ny} +u_n^2 u_{nxx} - u_n u_{nx}^2)
    \qquad \text{ at $z_n$.}
  \end{align*}
  Taking limits by use of  \eqref{eqn:cleanconv} then yields
  $u_{xx}(z_0) = 0$.  Thus by assumption the point $z_0$ could not be the trough 
  and so  could only be the crest.  
  Now plugging $u_{xx} = u_x = v = 0$ into \eqref{eqn:ptwo}
  evaluated at the crest, we obtain the equality 
  \begin{align}
    \label{eqn:factored}
      0 &= -2u^2u_y^2 - guu_y -\gamma(3u^2 u_y + gu) - \gamma^2 u^2  
      = -u^2 \eta_{xx} (2u^2\eta_{xx} - \gamma u + g)
  \end{align}
  there, where we have substituted the formula $\eta_{xx} = (u_y+\gamma)/u$, 
  which is due to $u_y+\gamma = v_x = (\eta_x u)_x = \eta_{xx}u$ at the crest.  
  By Theorem~\ref{thm:pressure}\ref{thm:pressure:curvature},
  $\eta_{xx} < 0$ at the crest.  So \eqref{eqn:factored} implies
  \begin{align*}
    g - \gamma u\crest
    = -2u^2\eta_{xx}\crest > 0,
  \end{align*}
  which is a contradiction.
\end{proof}

\begin{corollary}\label{cor:sigma}
  Under the same conditions as in Theorem \ref{thm:main}, every
  nontrivial wave in $\cm$ satisfies $u_x < 0$ in $\Omega^+ \cup S^+$,
  as well as 
  \begin{align*}
    \left| \frac vu\right| < \sigma  
    \qquad \textup{where} \ \ \sigma^2 = 
    \max_{\overline \Omega} \frac{g-\gamma u}{g+\gamma u}.
  \end{align*}
  Note that $\sigma^2<1$ if $\gamma(0)<0$.  
\end{corollary}
\begin{proof}
  We have already proven that $u_x < 0$ in $\Omega^+ \cup S^+$.  Now
  consider the function $f=\alpha u^2-v^2$ for some $0<\alpha \le 1$.
  Differentiating it and using  \eqref{eqn:pone} to eliminate $u_y$,
  we obtain 
  \begin{align} \label{alphaid}
    uf_x+vf_y  
    =  (\alpha+1)(v^2+u^2)u_x  
    -  \{\alpha(\gamma u+g) + \gamma u-g\} v.     
  \end{align}
  We define $\alpha  = \sigma^2$. Then, since $\gamma u+g>0$, the
  expression in \eqref{alphaid} is at most zero in $\Omega^+$. As in
  Lemma~\ref{lem:easy}, we deduce the stated inequality on the slope.  
\end{proof}

\section{Existence of waves}  \label{sec-existence}

We briefly discuss the question of existence of waves with vorticity
for which there is a bound on the slope.  In \cite{cs:exact} the
following construction of waves in $\w$ is proven.  Given $c$, $m$,
wavelength $2L$ and a smooth function $\gamma\le 0$, there exists a
connected set $\cm$ of waves satisfying \eqref{eqn:ww},
\eqref{eqn:nostag} and \eqref{eqn:mono} such that $\cm$ contains
exactly one trivial wave as well as a sequence of waves for which
$\sup u_n \nearrow c$.  Under the assumption $\gamma \le 0$, there are
no restrictions on $c$, $m$, and $L$ for the existence of $\cm$ (see
Example 3.4 in \cite{constantin:book}).  
(Note however that the definition $\gamma$ in that reference differs from ours by a sign.) 
Connectedness is taken in the same sense as above. (Actually in
\cite{cs:exact} the amount of regularity is less but the extra
regularity is very easy to prove.) In fact, $\cm$ contains a
continuous curve $\K$ in function space with the same properties.
This was proven later in the irrotational case in \cite{bt:analytic}
and in the rotational case with surface tension in
\cite{walsh:stratcap2}.

Before discussing the relationship between Theorem~\ref{thm:main} and
$\cm$, we make a few definitions related to Bernoulli's law.
First, we define a function $\Gamma \in C^3[-m,0]$ in terms of
$\gamma$ by
\begin{align*}
  \Gamma(-\psi) = \int_0^{-\psi}\gamma(-p)\, dp.
\end{align*}
In terms of $\Gamma$, which follows from \eqref{eqn:ww}, 
Bernoulli's law can be written
\begin{align}
  \label{eqn:Bernoulli}
  P - \patm + \frac{u^2+v^2}2 + g(y+d) - \Gamma(-\psi) \equiv Q
\end{align}
for some constant $Q$ sometimes called the ``total head''. For trivial
flows with fixed vorticity function $\gamma$ and flux $m$, it is
shown, for instance in \cite{cs:exact}, that the speed squared
$\lambda = u^2_0(0)$ at the free surface and the total head $Q$ are
related by $Q = \qt(\lambda)$, where 
\begin{align*}
  \qt(\lambda) := \frac \lambda 2 + 
  g\int_{-m}^0 \frac{ds}{\sqrt{\lambda + 2\Gamma(s)}}.
\end{align*}
We easily check $\qt$ is a strictly convex function of $\lambda >
-2\mn\Gamma$, with a unique minimum at $\lambda = \lc$.

Now let $\gamma$ satisfy \eqref{eqn:mastergam}, and let $\cms$ be the
set of waves in $\cm$ such that \eqref{eqn:trough} holds, which means
that $g-\gamma u >0$ at the trough.  

\begin{proposition}\label{prop:bif}\hfill
  \begin{enumerate}[label=\textup{(\alph*)}]
  \item\label{prop:bif:loc} $\cms$ contains $\cm_\loc$,
    the part of $\cm$ sufficiently close to the trivial wave.  
  \item\label{prop:bif:small} If $\gamma(0)$ is sufficiently
    small, namely, 
    \begin{align}
      \label{eqn:gammasmall}
      \gamma^2(0) < \frac {g^2}{2gL+ \lc}  
    \end{align}
    then $\cms = \cm$.  
  \end{enumerate}
\end{proposition}
\begin{proof}
  For \ref{prop:bif:loc} it suffices by continuity to prove that 
  \begin{align}
    \label{eqn:lambdagoal}
    g-\gamma \sqrt\lambda = g-\gamma u\trough > 0 
  \end{align}
  holds for the trivial wave in $\cm$. From \cite{cs:exact} we know
  that $\lambda < \lc$ for the trivial wave in $\cm$, so it
  is enough to show 
  \begin{align}
    \label{eqn:predesired}
    \gamma^2(0) \lc < g^2 .
  \end{align}
  If $\gamma(0) = 0$ then \eqref{eqn:predesired} is trivially
  satisfied, so assume that $\gamma(0) < 0$. 
  Then by the convexity of $\qt$ and the definition of $\lc$,
  \eqref{eqn:predesired} is equivalent to the inequality 
  \begin{align}
    \label{eqn:desired}
    0 =   2\qt'(\lc)  <            2\qt'\Big( \frac{g^2}{\gamma^2(0)} \Big) 
    = 1 - g\int_{-m}^0 
    \frac{dp}{(g^2/\gamma^2(0) + 2\Gamma(p))^{3/2}}.
  \end{align}
  Now by our assumptions \eqref{eqn:mastergam} on the vorticity
  function $\gamma$, $\Gamma'(0) = \gamma(0) < 0$, and moreover
  $\Gamma''(p) = -\gamma'(-p) \ge 0$ so that $\Gamma$ is convex and
  $\Gamma(p)\ge\gamma(0)p$.  Thus
  \begin{align*}
    2\qt'\Big( \frac{g^2}{\gamma^2(0)} \Big) 
    &\ge 
    1-g\int_{-m}^0 \frac{dp}{(g^2/\gamma^2(0) + 2\gamma(0)p)^{3/2}}\\
    &= \frac g{\abs{\gamma(0)}\sqrt{g^2/\gamma^2(0) - 2\gamma(0)m}} >
    0
  \end{align*}
  as desired.

  It remains to prove \ref{prop:bif:small}. Since $\cms \subset \cm$
  is nonempty and $\cm$ is connected, it suffices to show that $\cms$
  is both relatively open and relatively closed as a subset of $\cm$.
  From its definition (see \eqref{eqn:trough}), $\cms$ is clearly
  relatively open, so it remains to show that it is relatively closed.
  So consider a wave which is a limit point of $\cms$. Since waves in
  $\cms$ satisfy $\abs{\eta_x} < 1$,  this limiting wave must have
  $\abs{\eta_x} \le 1$. Evaluating Bernoulli's law
  \eqref{eqn:Bernoulli} both at the crest $\crest$ and the trough
  $\trough$, we deduce that
  \begin{align}
    \label{eqn:reduceme}
    u^2\trough = u^2\crest + 2g[\eta(0) - \eta(L)]  
    \le u^2\crest + 2gL.
  \end{align}
  But it was shown in \cite{cs:exact} that all the waves in $\cm$
  satisfy $u^2\crest < \lc$.   So \eqref{eqn:reduceme} implies
  $u^2\trough \le \lc + 2gL$. Rearranging this inequality and using
  the assumption \eqref{eqn:gammasmall}, we obtain $g-\gamma u\trough
  > 0$.  That is, the limiting wave in fact lies in $\cms$.
\end{proof}

Using the same argument as in the proof of
Proposition~\ref{prop:bif}\ref{prop:bif:loc}, one can
show that \eqref{eqn:gammasmall} holds whenever
\begin{align}
  \label{eqn:gammasmallest}
  \frac 1{\sqrt{1-2L\gamma^2(0)/g}}
  - \frac 1{\sqrt{1-2L\gamma^2(0)/g-2\gamma^3(0)m/g^2}} < 1.
\end{align}
Moreover, \eqref{eqn:gammasmall} and \eqref{eqn:gammasmallest} are
equivalent when $\gamma$ is constant.

\bigskip An open problem is the following question.  When $\gamma(0)$
is large enough that \eqref{eqn:gammasmall} is violated, does
\eqref{eqn:trough} still hold for all waves in $\cm$? In other words,
is $\cms=\cm$ always true?

\section{Solitary waves}  \label{sec-solitary}
In this short section we show how the arguments of
Sections~\ref{sec:periodic} and \ref{sec-existence} can be
modified for waves which are solitary rather than periodic. By a
solitary wave we mean a solution to \eqref{eqn:ww} (with $u=U-c,\
v=V$) in the unbounded domain
\begin{align*}
  \Omega = \{(x,y) \in \R^2 : -\infty < x < \infty,\ -d < y < \eta(x)\} 
\end{align*}
with the additional asymptotic condition that
\begin{align}
  \label{eqn:asym}
  D^k v \to 0, 
  \quad 
  \eta \to 0,
  \quad 
  D^k u(x,y) \to D^k u_\infty(y)
  \quad 
  \text{ as } x \to \pm\infty,
  \qquad k=0,1,2,
\end{align}
uniformly in $y$, for some function $u_\infty(y)$. Here $D^k$ denotes
any derivative of order $k$ with respect to $x$ or $y$.  As before we
will assume that $u,\eta$ are even in $x$ and that $v$ is odd in $x$.
We continue to assume \eqref{eqn:nostag}, $\sup_\Omega u < 0$, which
in particular implies $\max u_\infty < 0$. We will assume  the
regularity $\eta \in C^4_\bdd(\R)$, $u,v,P \in
C^3_\bdd(\overline{\Omega})$, and $u_\infty \in C^3[-d,0]$. The
notation $C^k_\bdd$ indicates functions whose derivatives up to order
$k$ are bounded and continuous.  The topology is that of uniform
convergence of those derivatives.  

Letting $\Omega^+$, $S^+$, and $B^+$ denote the right halves of
the fluid domain, free surface, and bed,
\begin{align*}
  \Omega^+ &= \{(x,y) \in \R^2 :  x > 0 ,\ -d < y < \eta(x)\},\\
  S^+ &= \{(x,\eta(x)) : x > 0  \},\\
  B^+ &= \{(x,-d) : x > 0 \},
\end{align*}
(see Figure~\ref{fig:omegaplus}b), we will continue to assume that the
strict monotonicity \eqref{eqn:mono} holds, that is $v > 0$ in
$\Omega^+ \cup S^+$, for all nontrivial waves. The vorticity function
$\gamma$ is defined exactly as before, so that $\gamma=-(u_\infty)_y$,
and will be assumed to satisfy \eqref{eqn:mastergam}. 

\begin{definition}[The water wave set $\w$ for solitary waves] 
  In this section we denote by $\w$  the set of waves satisfying the
  above assumptions, or more precisely the set of tuples
  $(u,v,P,\eta,u_\infty,d)$ with $d > 0$, $\eta \in C^4(\R)$, $u_\infty \in C^3_\bdd[-d,0]$, and $u,v,P \in
  C^3_\bdd(\overline\Omega)$ for which $u,P,\eta$ are even in $x$,  and $v$ are odd
  in $x$, which satisfy \eqref{eqn:ww}, \eqref{eqn:asym},
  \eqref{eqn:nostag}, and \eqref{eqn:mastergam}, and which either
  satisfy the monotonicity condition \eqref{eqn:mono} or are trivial.
\end{definition}

We also define the mapping $\tau$ in an analogous way, namely, 
\begin{align*}
  \tau \colon \w \to \big(C^0_\bdd[0,\infty)\big)^7,
  \qquad 
  \tau(u,v,P,\eta,u_\infty,d)(x) = (u,v,u_x,v_x,u_y,u_{xx},u_{xy})(x,\eta(x)).
\end{align*}

\begin{theorem}\label{thm:main:sol} 
  Let $\cm \subset \w$ be a set of solitary water waves such that
  $\tau(\cm)$ is a connected subset of $(C^0[0,\infty))^7$. If $\cm$
  contains at least one trivial wave, and if all the waves in $\cm$
  satisfy
  \begin{align}
    \label{eqn:trough:sol}
    g-\gamma u_\infty(0) > 0
  \end{align}
  then the slope $\abs{v/u}$ of any streamline of any wave in $\cm$ is $<1$.
\end{theorem}

Recall that $u_\infty(0)$ is the relative velocity of the fluid on the surface at infinity.  
The proof of this theorem will follow from the following lemmas.  

\begin{lemma}\label{lem:easy:sol}
  If a wave in $\w$  satisfies $u_x < 0$ in $\Omega^+ \cup S^+$, then
  $\sup_{\overline\Omega} \abs{v/u} < 1$.
\end{lemma}

\begin{proof} 
   As in Lemma \ref{lem:easy}, the function $f = \frac 12(u^2 - v^2)$
   has $f=u^2/2 > 0$ on $\pa\Omega^+\setminus S^+$.   By
   \eqref{eqn:asym}, $f$ is a  bounded function and $\lim_{x\to\infty}
   f  =  u_\infty^2/2 >0$.  Differentiating $f$, we find 
   \begin{align*}
     uf_x + vf_y = (u^2+v^2)u_x - \gamma uv < 0
     \quad \text{ on } \Omega^+ \cup S^+ .
   \end{align*}
   Suppose for the sake of contradiction that
   $\min_{\overline{\Omega^+}} f \le 0$.  Then the minimum could not
   be achieved on $\pa\Omega^+\setminus S^+$ nor at infinity, but only
   at some point $(x^*,y^*)$ in $\Omega^+ \cup S^+$.   The proof
   concludes exactly as in Lemma \ref{lem:easy}.  
\end{proof}

As in Section \ref{sec:periodic} we let $\w_0$ denote the set of waves
in $\w$ for which $u_x \le 0$ on $S^+$. The set $\w_0$ contains all
the trivial waves in $\w$, and $(u,v,P,\eta,u_\infty,d)$ in $\w$ lies in $\w_0$
if and only if $\tau(u,v,P,\eta,u_\infty,d)$ lies in $\tau(\w_0)$. 

\begin{lemma}[Bounds along $S^+$]\label{lem:bound:sol}
  \hfill
  \begin{enumerate}[label=\textup{(\alph*)}]
  \item\label{lem:bound:sol:le} For every nontrivial wave in $\w_0$, $u_x
    < 0$ on $\Omega^+ \cup S^+$.
  \item\label{lem:bound:sol:gtr} For every nontrivial wave in $\w
    \setminus \w_0$, $\inf_{\Omega^+} u_x/u < 0$ is achieved at a
    point $z_0 \in S^+$ where 
    \begin{align}
      g(u^4-4u^2v^2 -v^4) - \gamma u^3(u^2 + 5v^2) < 0.
    \end{align}
  \end{enumerate}
\end{lemma}

\begin{proof}  
  We follow the proof of Lemma~\ref{lem:bound}. Consider any
  nontrivial wave in $\w$. As before $w=u_x/u$ and $s=v/u$ satisfy the
  elliptic equations \eqref{eqn:w} and \eqref{eqn:s}, and applying the
  Hopf lemma to $s$ on $B^+$ yields $w < 0$ there. We also have $w=0$
  on $x=0$ by symmetry, and $w \to (u_\infty)_x/u_\infty = 0$ as $x
  \to +\infty$, uniformly in $y$, by \eqref{eqn:asym}. In particular,
  if $I = \inf_{\Omega^+} w < 0$, this infimum must be achieved at
  some point on $S^+$. The rest of the proof now proceeds exactly as
  in the proof of Lemma~\ref{lem:bound}.
\end{proof}

\begin{proof}[Proof of Theorem~\ref{thm:main:sol}]
  We follow the proof of Theorem~\ref{thm:main}. As before it suffices
  to show that $\tau(\cm) \cap \tau(\w_0)$ is relatively open in
  $(C^0_\bdd[0,\infty))^7$. Assume the contrary.  Then as before there
  would exist a sequence $(u_n,v_n,P_n,\eta_n,u_\infty,d_n) \in \cm \setminus
  \w_0$ and $(u,v,P,\eta,u_\infty,d) \in \cm \cap \w_0$ such that 
  \begin{align}
    \label{eqn:conv:sol}
    \tau(u_n,v_n,P_n,\eta_n,u_\infty,d_n) \longrightarrow \tau(u,v,P,\eta,u_\infty,d) 
    \quad 
    \text{ in $(C^0[0,\infty))^7$ }
  \end{align}
  Note that for each $n$, $\lim_{x\to\infty} (u_n)_x/u_n = 0$
  uniformly.  Applying
  Lemma~\ref{lem:bound:sol}\ref{lem:bound:sol:gtr} to
  $(u_n,v_n,P_n,\eta_n,d_n)$ for a fixed $n$, we know that the
  function $u_{nx}/u_n$ achieves a negative minimum over
  $\overline{\Omega^+_n}$ at a point $z_n = (x_n,\eta_n(x_n))\in
  S^+_n$, $x_n > 0$, at which we have 
  \begin{align}
    \label{eqn:takeme:sol}
    \Big[g(u^4_n-4u^2_nv^2_n -v^4_n) 
    - \gamma_n u^3_n(u^2_n + 5v^2_n)\Big](z_n) < 0.
  \end{align}

  Suppose first that $x_n \to +\infty$. Then the uniform convergence, 
  \eqref{eqn:conv:sol} and \eqref{eqn:asym}, imply that 
  \begin{align}
    \label{eqn:cleanconv:solinf}
    (u_n,v_n,u_{nx},v_{nx},u_{ny},u_{nxx},u_{nxy})(z_n)
    \longrightarrow (u_\infty(0),0,0,0,(u_\infty)_y(0),0,0)
  \end{align}
  because 
  \begin{align*}
    \abs{u_n(x_n,y_n) - u_\infty(0)}
    &\le 
    \abs{u_n(x_n,\eta_n(x_n)) -u(x_n,\eta(x_n))}
    + \abs{u(x_n,\eta(x_n)) - u_\infty(0)}\\
    &\le 
    \sup_x \abs{u_n(x,\eta_n(x))-u(x,\eta(x))}
     + \abs{u(x_n,0) - u_\infty(0)}
     + C \abs{\eta(x_n)},
  \end{align*}
  and similarly for the other components in
  \eqref{eqn:cleanconv:solinf}. Taking $n \to \infty$ in
  \eqref{eqn:takeme:sol} therefore yields $g^4(0)-\gamma u_\infty^5(0)
  \le 0$ and hence $g-\gamma u_\infty(0) \le 0$, contradicting
  \eqref{eqn:trough:sol}. 

  So $x_n$ must be bounded.  By compactness we
  may now assume that $x_n$ converges to some $x_0 \ge 0$. The uniform
  convergence \eqref{eqn:conv:sol} then implies that that 
  \begin{align}
    \label{eqn:cleanconv:sol}
    (u_n,v_n,u_{nx},v_{nx},u_{ny},u_{nxx},u_{nxy})(z_n)
    \longrightarrow (u,v,u_x,v_x,u_y,u_{xx},u_{xy})(z_0)
  \end{align}
  where $z_0 = (x_0,\eta(x_0))$ lies on the free surface. In case
  $(u,v,P,\eta,u_\infty,d)$ is trivial, \eqref{eqn:cleanconv:sol}
  reduces to \eqref{eqn:cleanconv:solinf} and we get a contradiction
  as in the preceding paragraph. So $(u,v,P,\eta,u_\infty,d)$ must be
  nontrivial. As in the proof of Theorem~\ref{thm:main}, we have $0 <
  u_{nx}(z_n) \to u_x(z_0) \le 0$ as $n\to\infty$ so that $u_x(z_0) =
  0$. Lemma~\ref{lem:bound:sol}\ref{lem:bound:sol:le} now implies that
  $z_0$ must be the crest $\crest$. Applying Theorem
  \ref{thm:pressure:sol}\ref{thm:pressure:sol:top}, we obtain
  \begin{align*}
    \frac d{dx} u^2(x,\eta(x)) = 2uu_x + 2vu_y =  -2{P_x} > 0
    \quad \text{for $x > 0$}.
  \end{align*}
  Thus $0 > u\crest > u_\infty(0)$ and hence
  \begin{align*}
    g-\gamma u_\infty(0) \le g - \gamma u \crest = g-\gamma u(z_0) \le 0,
  \end{align*}
  contradicting \eqref{eqn:trough:sol}.
\end{proof}

\begin{corollary}\label{cor:sigma:sol}
  Under the same conditions as in Theorem \ref{thm:main:sol}, every
  nontrivial wave in $\cm$ satisfies $u_x < 0$ in $\Omega^+ \cup S^+$,
  as well as 
  \begin{align*}
    \left| \frac vu\right| < \sigma  
    \qquad \textup{where} \ \ \sigma^2 = 
    \max_{\overline \Omega} \frac{g-\gamma u}{g+\gamma u}.
  \end{align*}
\end{corollary}
\begin{proof}
  The proof is exactly the same as the proof of
  Corollary~\ref{cor:sigma} for periodic waves, with
  Lemma~\ref{lem:easy} replaced by Lemma~\ref{lem:easy:sol}.
\end{proof}

In \cite{wheeler:solitary,wheeler:froude}, the following construction
of solitary waves in $\w$ is proven.  Given $c$, $d$, and a smooth
negative function $u_\infty^*$ of $-d \le y \le 0$ with
$(u_\infty^*)_y \ge 0$ and satisfying the normalization condition
\begin{align*}
  g \int_{-d}^0 \frac{dy}{(u_\infty^*)^2(y)} = 1,
\end{align*}
there exists a connected set $\cm$ of waves satisfying
\eqref{eqn:ww}, \eqref{eqn:nostag}, \eqref{eqn:mono}, and
\eqref{eqn:asym} such that $\cm$ contains exactly one trivial wave as
well as a sequence of waves for which $\sup u_n \nearrow c$. Here the
asymptotic horizontal velocity $u_\infty(y)$ in \eqref{eqn:asym} is
given by 
\begin{align*}
  u_\infty(y) = Fu_\infty^*(y) 
\end{align*}
for some positive non-dimensional parameter
$F$, called the Froude number,  which varies along $\cm$. The trivial wave in $\cm$ has $F
= 1$, while the nontrivial waves in $\cm$ have $1 < F < 2$.

Now suppose that $u_0 = u_\infty^*$ satisfies \eqref{eqn:masteru0}, in
which case the vorticity function $\gamma$ of any wave in $\cm$
satisfies \eqref{eqn:mastergam}. Let $\cms$ be the set of waves in
$\cm$ such that \eqref{eqn:trough:sol} holds, which means that 
$g-\gamma u_\infty(0) >0$.

\begin{proposition}\label{prop:bif:sol}\hfill
  \begin{enumerate}[label=\textup{(\alph*)}]
  \item\label{prop:bif:sol:loc} $\cms$ contains $\cm_\loc$,
    the part of $\cm$ sufficiently close to the trivial wave.  
  \item\label{prop:bif:sol:F} $\cms$ contains all waves in $\cm$
    satisfying the bound
    \begin{align}
      \label{eqn:Fsmall}
      F^2 < \frac g{\abs{(u_\infty^*)_yu_\infty^*}(0)}.
    \end{align}
  \item\label{prop:bif:sol:gamma} If $u_\infty^*$ satisfies
    \begin{align}
      \label{eqn:Fallsmall}
      \abs{(u_\infty^*)_y u_\infty^*}(0)  < \frac g4
    \end{align}
    then \eqref{eqn:Fsmall} always holds, so that $\cms = \cm$.
  \end{enumerate}
\end{proposition}
\begin{proof}
  For \ref{prop:bif:sol:loc} it suffices by continuity to prove that
  \eqref{eqn:trough:sol} holds for the trivial wave in $\cm$. From
  \cite{wheeler:solitary} we know that this wave has, in the notation
  of Section~\ref{sec-existence}, $\lambda = \lc$, so the
  rest of the proof proceeds exactly as in the proof of
  Proposition~\ref{prop:bif}\ref{prop:bif:loc}.
  To prove~\ref{prop:bif:sol:F}, we simply notice that, by the scaling
  $u_\infty = Fu_\infty^*$ and the definition of $\gamma$, we have 
  \begin{align*}
    g-\gamma u_\infty(0)
    = g+[(u_\infty)_y\, u_\infty](0)
    = g+F^2[(u_\infty^*)_y\, u_\infty^*](0) 
    = g - F^2\abs{(u_\infty^*)_y(0)\,u_\infty^*(0)}.
  \end{align*}
  The remaining statement \ref{prop:bif:sol:gamma} then follows
  immediately from \ref{prop:bif:sol:F} and the inequality $F < 2$ in
  \cite{wheeler:froude}.
\end{proof}
In the case of constant vorticity, the dimensionless vorticity
$\tilde\gamma := \gamma d/\abs{u_\infty(0)}^{1/2} < 0$ is constant
along $\cm$, and \eqref{eqn:Fallsmall} is equivalent to
$\abs{\tilde\gamma} < 1/3$.

\section{Some general inequalities on the pressure} \label{sec-pressure}
\subsection{Periodic case} \label{sec-pressure:periodic}

Some of the following facts are already known under various
assumptions but others appear to be new.  In fact, under certain
assumptions, versions of \ref{thm:pressure:basic} and
\ref{thm:pressure:moretop} appear in \cite{varv:conjecture},  and
a version of  \ref{thm:pressure:bed} appears in
\cite{constantin:heights}.

In this section, for the sake of generality, we assume only that the
free surface is a $C^2$ curve which does not self-intersect, along
which $y>-d$, and which is horizontally periodic with period $2L$. We
let $\Omega$ be the region between $S$ and $B=\{y=-d\}$, and require
that the equations \eqref{eqn:ww} hold in $\Omega$, with the kinematic
boundary condition \eqref{eqn:ww:top} replaced by the condition that
$(u,v)$ is tangent to $S$. Here we continue to denote $u=U-c,\ v=V$
and we define the vorticity to be $\omega = v_x - u_y$. We assume the
regularity and periodicity $u,v \in C^1_\per(\overline\Omega)$ and $P
\in C^2_\per(\overline\Omega)$, where as before ``$\per$'' denotes
$2L$-periodicity in $x$. 

We do {\it not} assume that $S$ is a graph.  Thus {\it the waves could
be overturning}.  Nor do we assume that the wave is symmetric in any
way, or that \eqref{eqn:nostag} or the monotonicity conditions
\eqref{eqn:mono} hold. Finally, we {\it not} assume that $\omega$ is a
function of $\psi$, nor that it has any particular sign.

Define
\begin{align*}
 \mx\eta = \max_S y, \qquad \mn\eta= \min_S y. 
\end{align*}
Any point on $S$ for which $y=\mx\eta$ is called a {\it crest}, while
any point on $S$ for which $y=\mn\eta$ is called a {\it trough}.

As before, incompressibility allows us to define a stream function
$\psi$, which by the boundary conditions can be taken to vanish on the
free surface $S$ and to be equal to some other constant $m$, not
necessarily positive, on the bed $B$.

\begin{theorem}\label{thm:pressure}
  Consider any nontrivial solution (that is, with $S$ not a horizontal
  line) to \eqref{eqn:ww} in $\Omega$ in the above sense with $u^2
  +v^2 \ne 0$.
  \begin{enumerate}[label=\textup{(\alph*)}]
  \item \label{thm:pressure:basic} 
    The pressure satisfies 
    \begin{align}
      \label{eqn:pressure:basic}
      g\mn \eta \le P-\patm+gy \le g\mx\eta,
    \end{align}
    with equality only at crests or troughs.
  \item \label{thm:pressure:curvature} 
    The free surface is strictly concave at any crest and  strictly
    convex at any trough.
  \item \label{thm:pressure:bed}
    $\mx \eta - \mn \eta > \frac 1g (\max_B P - \min_B P)$.
  \item \label{thm:pressure:below}
    $P > \patm$ at all depths below the troughs.
  \item \label{thm:pressure:top} 
    If $\omega u + g \ge 0$, then $P \ge
    \patm$ with equality only on the free surface $S$, on which $\pa P/\pa n <0$. 
  \item \label{thm:pressure:moretop} 
    If $\mx\omega (u^2+v^2) -4gu \ge 0$, 
    then $P + \frac 12 \mx \omega \psi \ge \patm$ with
    equality only on the free surface. 
    For instance, this is true if   $\mx\omega > 0$ and $u \le 0$.
  \item \label{thm:pressure:pos} 
    If $\omega \ge 0$ in the fluid, while $u < 0$ at either a crest or a trough,
    then the relative speed $\sqrt{u^2+v^2}$ in the fluid is maximized
    at all troughs.  
  \item \label{thm:pressure:neg} 
    If $\omega \le 0$ in the fluid, while $u<0$ at either a crest or a trough,
    then the relative speed $\sqrt{u^2+v^2}$ in the fluid is minimized
    at all crests.
  \end{enumerate}
\end{theorem}

\begin{proof}
  \ref{thm:pressure:basic} 
  Consider the function $f = P+gy$. Since $f = \patm + gy$ on the
  free surface, clearly it is nonconstant. A straightforward though
  tedious calculation shows that 
  \begin{align}\label{P+gy eq}
    (u^2+v^2) \Delta f 
    + 2(f_x - \omega v) f_x + 2(f_y + \omega u) f_y = 0.
  \end{align}
  Since $u^2+v^2 \ne 0$, the maximum principle implies that $f$ can
  only achieve its maximum and minimum values on the boundary. Since
  on the bed $y=-d$ we have $v=0$ and 
  \begin{align*}
    f_y = P_y + g = - uv_x - vv_y = 0,
  \end{align*}
  the Hopf lemma implies that the maximum and minimum of $f$ must be
  achieved on the free surface. But the pressure $P$ is constant on
  the free surface, so $f$ is maximized at all crests  
  and minimized at all troughs. 
  Rewriting this in terms of the pressure, we
  obtain \eqref{eqn:pressure:basic}  as desired.

  \ref{thm:pressure:below} 
  From the lower bound in \eqref{eqn:pressure:basic}, we see that
  \begin{align*}
    P-\patm \ge g(\mn\eta - y),
  \end{align*}
  and hence in particular that $P > \patm$ at all depths
  below the troughs, which is \ref{thm:pressure:below}.  

  \ref{thm:pressure:bed}
  Evaluating the inequalities in \eqref{eqn:pressure:basic} at the
  points along the bed where the pressure $P$ is maximized and
  minimized, we also find 
  \begin{align*}
    \textstyle\max_B P - \patm - gd < g\mx\eta, 
    \qquad 
    \min_B P - \patm - gd > g\mn\eta.
  \end{align*}
  Subtracting these two inequalities yields \ref{thm:pressure:bed}. 

  \ref{thm:pressure:curvature}
  Consider any crest.  In a neighborhood of the crest, $y$ is
  locally solvable as a function of $x$.  [Indeed, let $S$ be
  parametrized as $(\alpha(s),\beta(s))$ where $(\alpha')^2 +
  (\beta')^2 \ne 0$.  At the crest $\beta'=0$ so that $\alpha' \ne
  0$.  So we can locally solve $y = \beta(\alpha^{-1}(x)) =:
  \eta(x)$.]  Now at the crest $\eta_x = 0$  so we must have $v = 0$
  and $u \ne 0$ due to our assumption $u^2+v^2 \ne 0$. Applying the
  Hopf lemma to $f$ at the crest, we find that 
  \begin{align*}
    0 < f_y = P_y + g = -uv_x =  -u^2\eta_{xx} 
  \end{align*}
  and hence that $\eta_{xx} < 0$.  Thus the free surface is strictly
  concave at that point.   An analogous argument shows that
  $\eta_{xx} > 0$ at all troughs.  The emphasis here is on the
  strictness of the inequalities.  

  \ref{thm:pressure:top}  
  We note from \eqref{P+gy eq} and the assumption $\omega u + g \ge
  0$ that 
  \begin{align*}
    (u^2+v^2)\Delta P + 2(P_x-\omega v)P_x + 2(P_y + \omega u + 2g)P_y =
    -2g(\omega u + g) \le 0 
  \end{align*}
  in $\Omega$ and hence that $P$ achieves it minimum either on the
  free surface or on the bed $y=-d$. On $y=-d$ we have $P_y = -g <
  0$, so the minimum cannot occur there. Thus $P$ is minimized on
  the free surface, where it is constant. By the Hopf lemma, $\pa
  P/\pa n <0$ along the free surface $S$ where the normal $n$ points
  away from $\Omega$.

  \ref{thm:pressure:moretop}
  Consider the function $f = P + \frac 12 \mx \omega \psi$. Yet
  another direct and somewhat tedious calculation shows that $f$
  satisfies
  \begin{align*}
    2(u^2+v^2) \Delta f + 4af_x + 4bf_y
    = - (\mx\omega-\omega)(\mx\omega(u^2+v^2)-4gu) \le 0,
  \end{align*}
  where the coefficients $a$ and $b$ are 
  \begin{align*}
    a = f_x+(\mx\omega-\omega)v,
    \qquad 
    b = f_y-(\mx\omega-\omega)u+2g.
  \end{align*}
  So by the maximum principle $f$ can only be minimized on the
  free surface or the bed $y=-d$. On the bed, 
  \begin{align*}
    f_y = P_y + \tfrac 12 \mx \omega v = -g < 0,
  \end{align*}
  so the minimum cannot occur there. Thus $f$ is minimized on the
  free surface where it takes the constant value $\patm$.

  \ref{thm:pressure:pos} and  \ref{thm:pressure:neg}. 
  By our assumption $u^2+v^2 \ne 0$, the stream function $\psi$ has 
  no critical points in $\Omega$, so it must attain its maximum
  or its minimum on the free boundary, where it is constant. By
  assumption there is a crest or trough where $u < 0$, so that $\psi_x
  = v = 0$ there while $\psi_y = \pa \psi / \pa n = u < 0$. Thus
  $\psi$ must be minimized along the free surface where it vanishes, 
  and maximized along the bed where its value is $m > 0$.
  Now consider the function $\Gamma(x,y)$ defined by 
  \begin{align*}
    \Gamma(x,y) = P - \patm + \frac{u^2+v^2}2 + gy - K.
  \end{align*}
  This is a generalization of the function $\Gamma$ of Section \ref{sec-existence}.
  Note  that $  \Gamma_x = \omega v, \ \Gamma_y = -\omega u, $ and
  hence $u\Gamma_x + v\Gamma_y = 0$.  Thus $\Gamma$ is constant along
  streamlines. In particular, $\Gamma$ is constant along the free
  surface and also along the bed, and we can choose the constant $K$
  so that $\Gamma$ vanishes on the free surface. 

  We claim that if $\omega\ge0$ in $\Omega$, then $\Gamma\le0$ in
  $\Omega$, and similarly if $\omega \le 0$ then $\Gamma\ge 0$. To
  prove this, choose any $z_0 \in \Omega$.  Let $\psi_0 = \psi(z_0)$
  and solve the ODE
  \begin{align*}
    \dot z(s) = \frac{\nabla \psi}{\abs{\nabla \psi}^2}(z(s)) ,
    \qquad z(0) = z_0,
  \end{align*}
  where the denominator $\abs{\nabla \psi}^2 = u^2+v^2 \ne 0$. We
  easily check that $\frac d{ds} \psi(z(s)) = 1$, so in particular  $
  \psi(z(-\psi_0))=0$ and $\psi(z(m-\psi_0)) = m$. That is,
  $z(-\psi_0)$ lies on the free surface while $z(m-\psi_0)$ lies on
  the bed. Thus we must have $\Gamma(z(-\psi_0)) = 0$.  Hence
  \begin{align*}
    \Gamma(z_0) &= \Gamma(z(0)) 
    = \Gamma(z(-\psi_0)) + \int_{-\psi_0}^0 \frac d{ds} \Gamma(z(s))\, ds 
    = \int_{-\psi_0}^0 \nabla \Gamma 
    \cdot \frac{\nabla \psi}{\abs{\nabla \psi}^2}(z(s))\, ds \\ 
    &=  \int_{-\psi_0}^0 
    \frac{\omega v\psi_x -\omega u\psi_y}
    {\abs{\nabla \psi}^2}(z(s))\, ds 
    = -\int_{-\psi_0}^0 \omega(z(s))\, ds.
  \end{align*}
  Thus $\Gamma(z_0) \ge 0$ if $\omega \le 0$, while $\Gamma(z_0) \le 0$
  if $\omega \ge 0$.
  
  Now let $z^*=(x^*,\eta(x^*))$ be a trough, so that $\Gamma(z^*)=0$.
  Suppose $\omega \ge 0$ in $\Omega$, so that $\Gamma \le 0$.  Then at
  any point $(x,y)\in \overline\Omega$ we have from Bernoulli's law
  that
  \begin{align*}
    (u^2 + v^2)(x,y) 
    &\le (u^2 + v^2 - 2\Gamma(-\psi))(x,y)  
      = C-2P-2gy    \\
     & \le C-2\patm -2g\eta(x^*) 
   \le (u^2+v^2 - 2\Gamma(-\psi))(z^*)  
    = (u^2+v^2)(z^*)   
  \end{align*}
  since $\Gamma(-\psi(0))=\Gamma(0)=0$.  Similarly, if $z^*$ is a
  crest and  $\omega \le 0$ in $\Omega$, then $\Gamma \ge 0$, so that
  at any point $(x,y)\in \overline\Omega$ we have
  \begin{gather*}
    (u^2 + v^2)(x,y) 
    \ge (u^2 + v^2 - 2\Gamma(-\psi))(x,y)  
    \ge (u^2+v^2 - 2\Gamma(-\psi))(z^*) 
    = (u^2+v^2)(z^*).\qedhere 
  \end{gather*}
\end{proof}

\begin{proposition}[Overturning periodic waves must have a pressure sink] 
  \label{prop:sink}
  Consider a periodic wave as above and suppose that $u^2+v^2 \ne 0$
  on the free surface. If the wave overturns, meaning that $u$ takes
  both positive and negative values along the free surface $S$, then
  there is a point on the free surface where $\pa P/\pa n > 0$.  In
  particular, the pressure $P$ achieves its minimum value inside the
  fluid and not on the free surface. Thus by
  Theorem~\ref{thm:pressure}\ref{thm:pressure:top} there is either a
  stagnation point in the fluid or a point where $\omega u + g < 0$.
  \begin{proof}
    Since $u^2+v^2 \ne 0$ along $S$, the vector $(v,-u)$ is normal to
    $S$ and either points into the fluid or out of the fluid. Without
    loss of generality we assume below that it points out of the
    fluid, which implies that $u < 0$ at crests and troughs. Also
    since $u^2+v^2 \ne 0$ along $S$, we can think of the free surface
    as being parametrized by a curve $(x,y) = (X(s),Y(s))$, where $X$
    and $Y$ are solutions of the ordinary differential equation $X'(s)
    = u(X(s),Y(s))$, $Y'(s) = v(X(s),Y(s))$.  Since the wave is
    periodic, $X'$ and $Y'$ are periodic functions of $s$ with some
    period $T$. Defining the angle $\theta(s)$ that the free surface
    makes with the horizontal in the usual way, we have
    \begin{align*}
      \cos\theta = \frac{X'}{\sqrt{(X')^2+(Y')^2}},
      \qquad 
      \sin\theta = \frac{Y'}{\sqrt{(X')^2+(Y')^2}}.
    \end{align*}
    Because of the periodicity of $X'$ and $Y'$, we know that
    $\theta(s+T)=\theta(s) + 2\pi m$ for some integer $m$, which must
    be constant by continuity. In fact, since the free surface does
    not self-intersect, this constant $m$ must vanish, so that
    $\theta$ is periodic. This can be seen by considering the tangent
    angle of the simple closed curve formed by drawing a horizontal
    line between two consecutive crests. 
    Differentiating along the free surface, we easily check that 
    \begin{align}
      \label{eqn:notethat}
      X'' = uu_x+vu_y = -P_x,
      \qquad 
      Y'' = uv_x+vv_y = -P_y-g,
      \qquad 
      \theta' = \frac{Y'X''-X'Y''}{(X')^2+(Y')^2}.
    \end{align}

    Now consider a wave which overturns and assume for the sake of
    contradiction that $\pa P/\pa n \le 0$ along the whole free
    surface. Since the normal vector $(v,-u)$ on $S$ points out of the
    fluid, we must have 
    \begin{align*}
      0 
      &\ge vP_x-uP_y 
      = -Y'X'' + X'Y''+gX'\\
      &= -((X')^2+(Y')^2) \theta' + g\sqrt{(X')^2+(Y')^2} \cos \theta
    \end{align*}
    all along the free surface by \eqref{eqn:notethat}. Rearranging,
    we find
    \begin{align}
      \label{eqn:rhslater}
      \theta' \ge g\frac{\cos\theta}{\sqrt{(X')^2+(Y')^2}}.
    \end{align}
    In particular, along $S$ we have 
    \begin{align}
      \label{eqn:cospos} 
      \theta'(s) > 0 \text{ whenever } \cos\theta(s) > 0.
    \end{align}
    Because the wave overturns, there exists $s_0$ for which
    $u(X(s_0),Y(s_0))>0$ and hence $\cos\theta(s_0)>0$. Define
    \begin{align*}
      s_1 = \inf\{s > s_0 : \theta(s) = \theta(s_0) \}.
    \end{align*}
    By periodicity, $s_1 \le s_0 + T$. Since $\cos\theta(s_1) =
    \cos\theta(s_0) > 0$, \eqref{eqn:cospos} implies
    \begin{align*}
      \theta'(s_0) >0, \quad 
      \theta'(s_1) >0, 
    \end{align*}
    so that $s_1 > s_0$. Thus $\theta(s)$ is strictly increasing near
    both $s_0$ and $s_1$, so there must be another  root of the
    equation $\theta(s)=\theta(s_0)$ between $s_0$ and $s_1$.  This
    contradicts the definition of $s_1$.  
  \end{proof}
\end{proposition}

\subsection{Solitary case} \label{sec-pressure:solitary}

As above, we assume only that the free surface is a $C^2$ curve which
does not self-intersect and  along which $y>-d$, and let $\Omega$ be
the region between $S$ and $B=\{y=-d\}$. We require that the
equations \eqref{eqn:ww} hold in $\Omega$, with the kinematic boundary
condition \eqref{eqn:ww:top} replaced by the condition that $(u,v)$ is
tangent to $S$. Instead of horizontal periodicity, we require that $S$
is a graph $y=\eta(x)$ for $\abs x$ sufficiently large and that 
\begin{align}
  \label{eqn:asym:weak}
  \eta \to 0, \quad P-\patm+gy \to 0 \qquad \textup{ as } x \to \pm\infty,
\end{align}
uniformly in $y$. The asymptotic condition \eqref{eqn:asym:weak} is
implied by \eqref{eqn:asym}. As for regularity, we assume that $u,v,P
\in C^1_\bdd(\overline{\Omega})$ while $P \in
C^2_\bdd(\overline\Omega)$.

We define 
\begin{align*}
  \mx\eta = \sup_S y, \qquad \mn\eta= \inf_S y.
\end{align*}
Any point on $S$ for which $y=  \mx\eta$ is called a {\it crest},
while any point on $S$ for which $y=\mn\eta$ is called a {\it trough}.
If $\mx\eta$ is achieved at $(\pm\infty,0)$, we also call
$(\pm\infty,0)$ a ``crest".  If $\mn\eta$ is achieved at
$(\pm\infty,0)$, we also call $(\pm\infty,0)$ a ``trough".

As before, incompressibility allows us to define a stream function
$\psi$, which by the boundary conditions can be taken to vanish on the
free surface $S$ and to be equal to some other constant $m$, not
necessarily positive, on the bed $B$.

\begin{theorem}\label{thm:pressure:sol}
  Consider any nontrivial solution to \eqref{eqn:ww} in $\Omega$ as
  above, with $u^2 +v^2 \ge \delta > 0$ for some constant $\delta$.
  Then
  \begin{enumerate}[label=\textup{(\alph*)}]
  \item \label{thm:pressure:sol:basic} 
    The pressure satisfies 
    \begin{align}
      \label{eqn:pressure:sol:basic}
      g\mn \eta \le P-\patm+gy \le g\mx\eta,
    \end{align}
    with equality only at finite crests or troughs (and in the limit
    as $x \to \pm\infty$).
  \item \label{thm:pressure:sol:curvature} 
    The free surface is strictly concave at any finite crest and
    strictly convex at any finite trough.   
  \item \label{thm:pressure:sol:bed}
    $\mx \eta - \mn \eta > \frac 1g (\sup_B P - \inf_B P)$.
  \item \label{thm:pressure:sol:below}
    $P > \patm$ at all depths below the troughs.
  \item \label{thm:pressure:sol:top} 
    If $\omega u + g \ge 0$, then $P \ge \patm$ with equality only on
    the free surface $S$, on which $\pa P/\pa n <0$. 
  \item \label{thm:pressure:sol:moretop} 
    Suppose that $\mx\omega (u^2+v^2) -4gu \ge 0$ and, in addition to
    \eqref{eqn:asym:weak}, we have 
    \begin{align}
      \label{eqn:asym:lessweak}
      P_y \to -g,
      \quad 
      v \to 0
      \qquad \textup{ as } x \to \pm\infty,
    \end{align}
    uniformly in $y$. Then $P + \frac 12 \mx \omega \psi \ge \patm$
    with equality only on the free surface. For instance, this is true
    if $\mx\omega > 0$, $u \le 0$, and \eqref{eqn:asym} holds.
  \item \label{thm:pressure:sol:pos} 
    If $\omega \ge 0$ in the fluid while $u < 0$ for $\abs x$
    sufficiently large, then the relative speed $\sqrt{u^2+v^2}$ in
    the fluid is maximized at all troughs.  
  \item \label{thm:pressure:sol:neg} 
    If $\omega \le 0$ in the fluid while $u<0$ for $\abs x$
    sufficiently large, then the relative speed $\sqrt{u^2+v^2}$ in
    the fluid is minimized at all crests.
  \end{enumerate}
\end{theorem}
\begin{proof}
  The proof is almost identical to the periodic case so we only
  indicate the differences.

  \ref{thm:pressure:sol:basic} By assumption, $f = P+gy \to \patm$ as
  $x\to\pm\infty$, uniformly in $y$. The Hopf lemma implies that the
  supremum and infimum of $f$ can only be achieved on the free surface
  or in the limit as $x \to \pm\infty$. If both are achieved on the
  free surface then the argument proceeds exactly as before, so
  suppose that the infimum of $f$ is achieved as $x \to \pm\infty$ but
  not at any finite point on the free surface. Then $f > \patm$ at all
  points in the fluid domain and along the free surface. Restricting
  $f$ to the free surface we see that $\mn\eta=0$ is also achieved as
  $x \to \pm\infty$. In particular, $(\pm\infty,0)$ is a trough by our
  above definition and the inequality $f > \patm$ can be rewritten as
  $P-\patm + gy > g\mn\eta$. Similarly, if the supremum of $f$ is
  achieved as $x \to \pm\infty$ but not at any point along the free
  surface then we must have $P-\patm + gy < g\mn\eta$.

  \ref{thm:pressure:sol:below} 
  No change.

  \ref{thm:pressure:sol:bed}
  Replace max and min by sup and inf. The inequality is still strict
  since by \eqref{eqn:asym:weak} the sup and inf of the restriction of
  $P$ to the bed $B$ cannot both be achieved as $x \to \pm\infty$.

  \ref{thm:pressure:sol:curvature}
  Consider any finite crest or trough.    

  \ref{thm:pressure:sol:top}  
  Replacing min by inf, suppose that the infimum of $P$ is achieved as
  $x \to \pm\infty$. Then $P + gy \to \patm$ as $x \to\pm\infty$
  implies $\inf P = \patm$ as before.

  \ref{thm:pressure:sol:moretop}  
  Thanks to \eqref{eqn:asym:lessweak}, we have 
  \begin{align*}
    f_y = P_y + \frac 12 \mx\omega v \to -g < 0 
  \end{align*}
  as $x \to \pm\infty$, uniformly in $y$. Thus if the infimum of $f$
  is achieved as $x \to \pm\infty$, it must be achieved at
  $(\pm\infty,0)$ where $f$ takes the value $\patm$.

  \ref{thm:pressure:sol:pos} and  \ref{thm:pressure:sol:neg}. 
  Since $\psi$ has no critical points and $\psi_y = u < 0$ for $\abs
  x$ sufficiently large, we find as before that $\psi$ is minimized
  along the free surface where it vanishes and maximized along the bed
  where its value is $m > 0$. If $z^*=(\pm\infty,0)$, we simply take
  limits.  
\end{proof}

\begin{proposition}[Overturning solitary waves must have a pressure sink] 
\label{prop:sink:sol}
  Consider a solitary wave as above and suppose that $u^2+v^2 \ge
  \delta > 0$ on the free surface for some constant $\delta$, and  $\eta_x
  \to 0$ as $x \to \pm\infty$. If the wave overturns, meaning that $u$
  takes both positive and negative values along the free surface $S$,
  then there is a point on the free surface where $\pa P/\pa n > 0$.
  In particular, the pressure $P$ achieves its minimum value inside
  the fluid and not on the free surface.
  \begin{proof}
    Again, the proof is almost identical to the periodic case so we
    only indicate the differences. 
    Assuming as before that $(v,-u)$ points out of the fluid, we have 
    $e^{i\theta(s)} \to -1$ as $s \to \pm\infty$.
    Without loss of generality, we may assume that $\theta(s) \to \pi$
    as $s \to +\infty$. To see that $\theta(s) \to \pi$ as $s \to
    -\infty$ as well, we consider the tangent angle of the oriented
    boundary of $\Omega \cap \{ \abs x < M \}$ for $M$ sufficiently
    large.

    Now consider a wave which overturns and assume for the sake of
    contradiction that $\pa P/\pa n \le 0$ along the whole free
    surface. As before we find
    \begin{align}
      \theta'(s) > 0 \text{ whenever } \cos\theta(s) > 0.
    \end{align}
    Since $\theta \to \pi$ as $s \to \pm\infty$, there must exist
    values $s_0<s_2$ for which $\theta(s_0) = \theta(s_2)$ and
    $u(X(s_0),Y(s_0)) > 0$ and hence
    $\cos\theta(s_0)=\cos\theta(s_2)>0$.  Let
    \begin{align*}
      s_1 = \inf\{s > s_0 : \theta(s) = \theta(s_0) \}.
    \end{align*}
    By construction, $s_1 < s_2$, so in particular $s_1$ is finite,
    and we reach a contradiction exactly as in the periodic case.
  \end{proof}
\end{proposition}


\end{document}